\documentclass[11pt]{amsart}

\usepackage{pb-diagram}

     \usepackage{xcolor}
\usepackage{amsfonts}
\usepackage{amsmath}
\usepackage{amssymb}
\usepackage{verbatim}
\usepackage{color}

\newcommand{\tr}{\textnormal{tr}}

\newcommand{\al}{\alpha}
\newcommand{\be}{\beta}
\newcommand{\ga}{\gamma}

\newcommand{\ep}{\epsilon}

\newcommand{\vp}{\varphi}

\newcommand{\vep}{\varepsilon}
\newcommand{\ob}{\omega_\be}
\newcommand{\om}{\omega}
\newcommand{\la}{\lambda}
\newcommand{\oeta}{\omega_{\eta}}
\newcommand{\vpe}{\varphi_\eta}

\newcommand{\peta}{\psi_\eta}
\newcommand{\deta}{d_{\eta}}

\newcommand{\dinf}{d_{\infty}}
\newcommand{\tvep}{\tilde\vep}
\newcommand{\oep}{\omega_\ep}
\newcommand{\vpep}{\varphi_\ep}

\newcommand{\oepet}{\omega_{\ep,\eta}}

\newcommand{\tomega}{\tilde{\omega}}

\newcommand{\dbar}{\overline{\partial}}

\newcommand{\sE}{\mathcal{E}}

\newcommand{\sR}{\mathcal{R}}
\newcommand{\sC}{\mathcal{C}}

\newcommand{\sL}{\mathcal{L}}

\newcommand{\homega}{\hat{\omega}}

\newcommand{\ddbar}{\sqrt{-1}\partial\dbar}

\newtheorem{theorem}{Theorem}[section]
\newtheorem{proposition}{Proposition}[section]
\newtheorem{lemma}{Lemma}[section]

\newtheorem{definition}{Definition}[section]
\newtheorem{corollary}{Corollary}[section]
\newtheorem{remark}{Remark}[section]

\newcommand{\CC}{\mathbb{C}}

\begin{document}
\pagestyle{plain}

\title{On convexity of the regular set of conical K\"ahler-Einstein metrics}

\author{Ved V. Datar$^*$}

\address{$*$ Department of Mathematics, Rutgers University, Piscataway, NJ 08854}

\email{veddatar@math.rutgers.edu}

\thanks{Research supported by the graduate fellowship of Rutgers University.}

\begin{abstract}In this note we prove convexity, in the sense of Colding-Naber, of the regular set of solutions to some complex Monge-Amp\`ere equations with conical singularities along simple normal crossing divisors. In particular, any two points in the regular set can be joined by a smooth minimal geodesic lying entirely in the regular set. As a consequence, the classical theorems of Myers and Bishop-Gromov extend almost verbatim to this singular setting.
\end{abstract}

\maketitle

\section{\bf Introduction}

Let $(X,\homega)$ be a K\"ahler manifold with a smooth reference K\"ahler metric $\homega$. A divisor 
\begin{equation}\label{divisor}
D =  \sum_{j=1}^{N}{(1-\be_j)D_j}
\end{equation}
 with $D_j$ a smooth irreducible divisor and $\be_j \in (0,1)$, is called a {\em simple normal crossing} divisor if locally at any $p\in D$ lying in the intersection of exactly $k$ divisors $D_1,\cdots,D_k$, there exists a coordinate chart $(U,(z_1,\cdots,z_n))$ such that $D_j\Big |_U$ is cut out by $[z_j=0]$ for $j=1,\cdots,k$. 
A conical K\"ahler metric along $D$ is a smooth K\"ahler metric on $X\setminus D$ which is locally equivalent to the following model edge metric
\begin{equation}\label{edge}
\omega_e = \sum_{j=1}^{k}{|z_j|^{-2(1-\be_j)}dz_j \wedge d\bar z_j}  + \sum_{j=k+1}^{N}{dz_j\wedge d\bar z_j}
\end{equation}
Of course, being a conical metric entails additional restrictions on the asymptotics near $D$; interested readers can refer to \cite{D,JMR} for more details. 

In this note, we are instead concerned with solutions to the following singular complex Monge-Amp\`ere equation :
\begin{equation}\label{ke ma}
\begin{cases}
(\homega + \ddbar\vp)^n = \frac{e^{-\la\vp}\Omega}{\prod_{j=1}^{N}{|s_j|_{h_j}^{2(1-\be_j)}}}\\
\omega = \homega+\ddbar\vp > 0 \\
\vp \in L^{\infty}(X) \cap PSH(X,\homega)
\end{cases}
\end{equation}
where $\la \in \mathbb{R}$,  $s_j$ is the defining section of $D_j$, $h_j$ is a smooth hermitian metric on the line bundle generated by $D_j$, and $\Omega$ is a smooth volume form satisfying 

\begin{equation}\label{coh}
\ddbar\log{\Omega} + \la\homega +\chi = \sum_{j=1}^{N}{(1-\be_j)\ddbar\log{h_j}}
\end{equation}
for some smooth $(1,1)$ form $\chi$. The Ricci curvature of $\omega$ solves a twisted conical K\"ahler-Einstein equation
\begin{equation}\label{con ke}
Ric(\om) = \la\om + \chi + [D]
\end{equation}
where $[D]$ is the current of integration along $D$.  A detailed study of such equations was first carried out by Yau \cite{Y2} in his seminal paper on the Calabi conjecture.  Applications of canonical conical K\"ahler metrics to Chern number inequalities were proposed by Tian in \cite{T94} (cf. also \cite{SW} for progress in this direction).  More recently, K\"ahler-Einstein metrics with cone singularities have played an important role in the breakthrough on the existence of K\"ahler-Einstein metrics on Fano manifolds \cite{CDS,T13}.   

It is well known, since Yau's work, that any bounded solution to \eqref{ke ma} is smooth on $X\setminus D$. In the special case when the divisor is $(1-\be)D$, i.e has only one smooth divisorial component, substantial progress has been made on understanding the behavior of the metric close to the divisor, starting with the fundamental linear theory of Donaldson in \cite{D}. Building on this, existence and regularity results were obtained by Brendle in the Ricci flat case when $\be\in (0,1/2)$, and by Jeffres-Mazzeo-Rubinstein \cite{JMR} for all $\be\in (0,1)$, and all signs of curvature. In \cite{JMR}, the authors also prove the existence of a polyhomogenous expansion for conical K\"ahler-Einstein metrics near the divisor.  An alternative proof of the existence results, relying on Donaldson's linear theory, also appears in the work of Chen-Donaldson-Sun \cite{CDS2}. 

Unfortunately, since many linear systems do not contain smooth divisors, it is important to address the questions of regularity for cone angles along normal crossing divisors. The first step in this direction are the results of Campana-Guenancia-P$\breve{\text{a}}$un \cite{CGP}, and Guenancia-P$\breve{\text{a}}$un \cite{GP}. Amongst other things, they prove that any $\omega$ solving \eqref{ke ma} is locally equivalent to the standard edge metric \eqref{edge} (cf. \cite{DS} for a shorter proof). While higher regularity results are awaited, the aim of this note is to prove the convexity of $X\setminus D$ with respect to the metric induced by $\omega$. This allows the extension of the classical comparison theorems to the conical setting, and might be useful in studying the moduli space of K\"ahler-Einstein metrics with cone singularities (cf. \cite{DGSW}).  

Since $\omega$ is smooth on $X\setminus D$, it defines a length functional $\mathcal{L}_{\omega}$ and in turn, a distance function
\begin{equation*}
 d_{\omega}(p,q) = \inf\{\sL_\omega(\ga) ~|~ \ga : [0,1]\rightarrow X\setminus D \text{ piecewise smooth }, ~ \ga(0)=p,\ga(1)=q\}
\end{equation*}

\bigskip

Since, $\omega$ is locally equivalent to a model edge metric \cite{CGP,GP,DS}, it is easily seen that the metric completion of $X\setminus D$ under this distance function is homeomorphic to $X$ itself, and we set $$(X,d) = \overline{(X\setminus D,d_\omega)}$$  where the bar denotes the metric completion. We first prove an approximation theorem for $\omega$, extending results of \cite{CDS1,T13} for conical K\"ahler-Einstein metrics in the Fano case.

\begin{proposition}\label{smoothening}
Let $\om = \homega+\ddbar\vp$ be a solution to \eqref{ke ma} with $\vp \in PSH(X,\homega)\cap L^{\infty}(X)$. Then there exist uniform constants $A,\Lambda \gg 1$, and a sequence $\oeta \in [\om]$ of smooth K\"ahler metrics such that \\
\begin{enumerate}
\item $Ric(\oeta) > -A\oeta$ ~; $diam(X,\oeta) < \Lambda$\\
\item As $\eta \rightarrow 0$, $$(X,\oeta) \xrightarrow{d_{GH}} (X,d)$$
where $(X,d)$ as above, is the metric completion of $(X\setminus D,d_\omega)$.
\end{enumerate}
\end{proposition}

It should be noted that in the Fano case with $\chi=0$ and a smooth pluri-canonical divisor $D$, Chen-Donaldson-Sun \cite{CDS1} and Tian \cite{T13}, prove a much stronger result, namely one can approximate with the {\em same} Ricci lower bound as the conical metric.  For such a result, it is of course necessary that $X$ is Fano.

Next, recall that a unit-speed path $\ga:[0,l] \rightarrow X$ joining $p,q$ is said to be a {\em minimal geodesic} if $d(p,q)=l$. It is said to be a {\em limiting geodesic} if there exists a sub-sequence $\{\eta_j\}$ with unit-speed $\omega_{\eta_j}$-geodesics  $\ga^{\eta_j}:[0,l_j]\rightarrow X$ such that $l_j \rightarrow j$ and $\ga^{\eta_j}\rightarrow \ga$ point wise. Limiting geodesics can be usually found in abundance. Our main theorem is :

\begin{theorem}\label{conv}
$X\setminus D \subset(X,d)$ is geodesically convex, in the following sense: if any interior point of a limiting minimal geodesic lies in $X\setminus D$, then all the interior points must lie in $X\setminus D$. 
\end{theorem}

The theorem is proved by combining the above smoothening with the results of Colding-Naber \cite{CN} on the H\"older continuity of tangent cones for limit spaces. It must be noted that the theorem does not rule out the possibility of {\em some}  geodesic connecting $p,q\in X\setminus D$ and passing through $D$, though it is expected that such a scenario will not occur. A nice consequence of the above theorem is the following

\begin{corollary}\label{smooth geod}
Let $p,q\in X\setminus D$ with $l = d(p,q)$. Then there exists a smooth unit speed geodesic $\ga:[0,l]\rightarrow X\setminus D$ with $\ga(0) = p$ and $\ga(l)= q$.
\end{corollary}

\begin{remark}
The theorem and the corollary may also hold under slightly milder conditions by replacing the smooth volume form $\Omega$ in \eqref{ke ma} by $e^{F}\homega^n$ with $F \in C^0(X)\cap C^{\infty}(X\setminus D)$, if we make an additional apriori assumption that $$Ric(\omega) > -C\omega$$ on $X\setminus D$.
\end{remark}

\noindent{\bf Notation.} Distances with respect to $\oeta$ and $\om$ are denoted by $\deta,d$ respectively. Paths connecting points $p,q$ are denoted by $\ga_{pq}$. Minimal geodesics are denoted with superscripts to specify the reference metric. For example, $d_\eta$-minimal and $d$-minimal geodesics are denoted by $\ga^{\eta}_{pq}$ and $\ga^{d}_{pq}$ respectively. The lengths of paths are denoted by $\sL_\omega$ and $\sL_\eta$ respectively. \\

\noindent{\bf Acknowledgements.} I am grateful to my advisor Jian Song for his support, encouragement and guidance throughout my graduate studies. I wish to thank Aaron Naber for clarifying that convexity in his work with Colding requires only a Ricci lower bound if the regular set can be shown to be open. I would also like to thank Xiaowei Wang, Zhenlei Zhang and Bin Guo for many stimulating discussions.

\section{\bf Approximation and smooth convergence away from $D$.}

We need to deal with the cases $\la \leq 0$ and $\la>0$ separately. 

\begin{itemize}

\item \textbf{$\la>0$.} We follow the same strategy as in \cite{CGP,CDS1,T13}. By Demailly's regularization theorem \cite{De}, there exists a sequence $\peta \in C^{\infty}(X)\cap PSH(X,\homega)$ such that $\peta \searrow \vp$ point wise as $\eta \rightarrow 0$. The metrics $\oeta$ are then constructed as the solutions to the following perturbation of \eqref{ke ma}
\begin{equation}\label{pert ma}
\begin{cases}
(\homega + \ddbar\vpe)^n =  \frac{e^{-\la\peta + c_\eta}\Omega}{\prod_{j=1}^{N}{(|s_j|^2_{h_j} + \eta)^{(1-\be_j)}}}\\
\oeta = \homega + \ddbar\vpe >0
\end{cases}
\end{equation}
 where $c_\eta$ is a constant such that the integrals on both sides are equal. The existence of a unique smooth solution for $\eta>0$, follows from the work of Yau \cite{Y2}. 
 
 \item\textbf{$\la \leq 0$.} In this case we let $\oeta = \homega+\ddbar\varphi_\eta$ be a solution of  
 \begin{equation}\label{pert ma easy}
\begin{cases}
(\homega + \ddbar\vpe)^n =  \frac{e^{-\la\vpe + c_\eta}\Omega}{\prod_{j=1}^{N}{(|s_j|^2_{h_j} + \eta)^{(1-\be_j)}}}\\
\oeta = \homega + \ddbar\vpe >0
\end{cases}
\end{equation}
The existence is guaranteed by Yau \cite{Y2} when $\la = 0$, and Aubin-Yau \cite{Au,Y2} when $\la<0$. When $\la=0$ we also need an additional normalization such as $\sup_X{\varphi_\eta} = 0$. Also note that $c_\eta$ can be taken to be zero when $\la<0$.
\end{itemize}
Note that $|c_\eta|$ is uniformly bounded, and in fact tends to zero as $\eta\rightarrow 0$. We next obtain lower bounds for the Ricci curvatures of $\oeta$.

\begin{lemma}\label{ric bdd weak}
If $\oeta = \homega + \ddbar\vpe$ is a solution to \eqref{pert ma}, then there exists an $A\gg1$ such that
\begin{enumerate}
\item When $\la>0$, $$Ric(\oeta) > - A\homega$$
\item When $\la\leq 0$, $$Ric(\oeta) > \la\oeta - A\homega$$
\end{enumerate}
\end{lemma}

\begin{proof}We follow the computation in \cite{CDS1}. First observe that for any smooth $f > 0$
\begin{align*}
\ddbar\log{(f + \eta)} \geq \frac{f}{f+\eta}\ddbar\log{f}
\end{align*}
On $X\setminus D$, using this and the fact that $1-\be_j\geq0$, we obtain  
\begin{align*}
\chi_\eta &:= -\la\homega - \ddbar\log{\Omega} + \sum_{j=1}^{N}{(1-\be_j)\ddbar\log{(|s_j|^2_{h_j}+\eta)}}\\
&\geq -\ddbar\log{\Omega} + \sum_{j=1}^{N}{(1-\be_j)\frac{|s_j|^2_{h_j}}{(|s_j|^2_{h_j}+\eta)}\ddbar\log{|s_j|^2_{h_j}}}\\
&= -\sum_{j=1}^{N}{(1-\be_j)\frac{\eta}{(|s_j|^2_{h_j}+\eta)}\ddbar\log{h_j}} + \chi \hspace{1in} \text{(by equation \ref{coh})}\\ 
&\geq -A\homega
\end{align*}
for some $A>>1$. Now, if $\la>0$, by \eqref{pert ma} $$Ric(\oeta) = \la(\homega+\ddbar\psi_\eta) + \chi_\eta > -A\omega$$ On the other hand, if $\la\leq 0$, by \eqref{pert ma easy}$$Ric(\oeta) = \la\oeta + \chi_\eta > \la\oeta - A\homega$$\end{proof}

Next, we obtain uniform $C^0$ and $C^2$ estimates on $\vpe$.

\begin{proposition}\label{estimate}
There exists a constant $C= C(n,\la,A,|| \om^n/\Omega||_{L^{1+\delta}(X,\Omega)}, ||Rm(\homega)||)\gg 1$ independent of $\eta$, such that \\

\begin{enumerate}
\item\begin{equation*}
||\vpe||_{C^0(X)} < C
\end{equation*}

\item \begin{equation}\label{c2}
C^{-1}\homega < \oeta < \frac{C\homega}{\prod_{j=1}^{N}{(|s_j|^2_{h_j} + \eta)^{(1-\be_j)}}}
\end{equation}
\end{enumerate}
\end{proposition}

\begin{proof} The proof is standard. We first assume that $\la>0$. The other case also follows similarly. The right hand side of equation \eqref{pert ma} is uniformly in $L^{1+\vep}(X,\omega)$ for some $\vep>0$, since all the $\be_j$'s are strictly positive, $\peta, |c_\eta|$ are uniformly bounded, and $D$ is a simple normal crossing divisor.  The $C^0$ estimate now follows directly from the work of Kolodziej \cite{K1,K2}. For the $C^2$ estimate, we consider the following quantity :
\begin{equation}
Q = \log{\tr_{\oeta}\homega} - B\vpe
\end{equation}
By Lemma \ref{ric bdd weak}, $Ric(\oeta)> -A\homega$ for some $A\gg 1$. Then by the Chern-Lu inequality, there exist constants $B,C\gg 1$ depending on $A$, the dimension $n$, and an upper bound for the bisectional curvature of $\homega$, such that 
\begin{equation*}
\Delta_\eta Q \geq \tr_{\oeta}\homega - C
\end{equation*}
By maximum principle and the uniform $C^0$ estimate, 
\begin{equation*}
\tr_{\oeta}\homega \leq C
\end{equation*}
But then using the equation \eqref{pert ma}, and an elementary arithmetic-geometric mean inequality
\begin{align*}
\tr_{\homega}\oeta &\leq (\tr_{\oeta}\homega)^{n-1}\frac{\oeta^n}{\homega^n}\\
&\leq \frac{C}{\prod_{j=1}^{N}{(|s_j|^2_{h_j} + \eta)^{(1-\be_j)}}}
\end{align*}
Note that when $\la\leq 0$, $Ric(\oeta)>\la\oeta -A\homega$, and Chern-Lu can still be applied. 
\end{proof}

Higher order estimates away from the divisor follow by standard methods (cf. \cite{PSS}). As a straightforward corollary we have,

\begin{corollary}
For the $\oeta$ constructed above 

\begin{enumerate}\label{smt conv}
\item There exists a uniform constant $A>>1$ such that $$Ric(\oeta) > -A\oeta$$
\item  Locally on $X\setminus D$, as $\eta \rightarrow 0$, $$\oeta \xrightarrow{C^{\infty}_{loc}(X\setminus D)} \om$$
\item For all open sets $U \subset X$, $$Vol_{\oeta}(U) \xrightarrow{\eta\rightarrow 0} Vol_{\omega}(U)$$
\end{enumerate}
\end{corollary}

Next, we prove a uniform diameter bound along the sequence $(X,\oeta)$. In the case of just one smooth divisorial component, this follows directly from the $C^2$ estimate in Proposition \ref{estimate} (cf. \cite{CDS1}). For the general simple normal crossing case, we instead use the work of Cheeger-Colding, and exploit the fact that the limiting conical metric has a bounded diameter.  

\begin{proposition}
There exists a $\Lambda > > 1$ such that $$diam(X,\deta)<\Lambda$$
\end{proposition}

\begin{proof}
By Gromov's compactness theorem \cite{Gr} $$(X,\deta,p) \xrightarrow{d_{GH}} (X_\infty,\dinf,p_\infty)$$ where $p\in X\setminus D$ is a fixed point, and $X_\infty$ is a metric length space. It follows from smooth convergence away from the divisor that $X\setminus D$ can be identified as an open subset of $X_\infty$. More precisely, there exists an injective local isometry $i : X\setminus D \rightarrow X_\infty$ such that 
\begin{equation}\label{Lip}
\dinf(i(x),i(y)) \leq d_{\omega}(x,y)
\end{equation}
and $i(X\setminus D) = X_0$ is an open subset of $X_\infty$. The Lipschitz bound above follows from the fact that the intrinsic distance is bigger than the extrinsic distance. Moreover, one has uniform non collapsing at $p$,  $$Vol(B_{\deta}(p,1))\geq \nu > 0$$ and hence by volume convergence \cite{CC1},
\begin{align*}
\mathcal{H}^{2n}_{\infty}(X_\infty) &= \lim_{\eta\rightarrow 0}Vol_{\oeta}(X)\\
&=\frac{[\homega]^n}{n!}
\end{align*} 
where $\mathcal{H}^{2n}_{\infty}$ denotes the $2n$-dimensional Hausdorff measure with respect to $\dinf$. Again due to smooth convergence away from $D$ and Cheeger-Colding volume convergence, the restriction of Hausdorff measure of $(X_\infty,\dinf)$ on $X_0$ agrees with the induced Riemannian measure from $X\setminus D$. But then, $$\mathcal{H}^{2n}_{\infty}(X_0) = \int_{X\setminus D}{\frac{\omega^n}{n!}} = \lim_{\eta\rightarrow 0}\int_{X}{\frac{\oeta^n}{n!}} = \frac{[\homega]^n}{n!} = \mathcal{H}^{2n}_{\infty}(X_\infty)$$ 
That is $X_0$ is of full measure inside $X_\infty$, and in particular, is dense. Then by the Lipschitz estimate \eqref{Lip} above, $diam(X_\infty,\dinf) \leq diam(X,d)$, where recall that $d$ is the distance induced by the conical metric $\omega$. But $(X,d)$ has finite diameter, since $\omega$ is quasi-isometric to a standard cone metric. Hence, $(X_\infty,\dinf)$ has a finite diameter, and since $(X,\deta,p) \xrightarrow{d_{GH}} (X_\infty,\dinf,p_\infty)$, $diam(X,d_\eta)$ is uniformly bounded.
\end{proof}

\section{\bf Almost geodesic convexity and Gromov-Hausdorff convergence.}

To complete the proof of Proposition \ref{smoothening}, we need to identify $X_\infty$ with the metric completion $(X,d)$ of $(X\setminus D, d_\omega)$. The proof follows along the lines of \cite{Z,T13,DGSW}. The main technical ingredient is the following relative comparison lemma of Gromov \cite{Gr}.

\begin{lemma}\label{gromov}
Let $(M,g)$, be a Riemannian manifold of dimension $m$ and $T\subset M$ be any compact set with a smooth boundary, such that  
\begin{itemize}
\item $$Ric(g)>-Ag~;~diam(M,g)<\Lambda$$
\item For some points $p_1,p_2 \in M$ with $B(p_j,\vep) \cap T = \emptyset$ for $j=1,2$, every minimal geodesic from $p_1$ to points in $B(p_2,\vep)$ intersects $T$.
\end{itemize}
Then, there exists a constant $c = c(n,\vep,A,\Lambda)$ such that 

\begin{equation*}
Vol(\partial T,g) > c Vol(B(p_2,\vep),g)
\end{equation*}
\end{lemma}

The next lemma proves that for almost all points, $X\setminus D$ is geodesic convex. This was proved by Cheeger-Colding \cite{CC2} for Gromov-Hausdorff limits of Riemannian manifolds with a Ricci lower bound. In our case, we haven't yet identified the Gromov-Haudorff limit with $X$, and so we give an elementary proof using the above comparison lemma and smooth convergence on $X\setminus D$.
\begin{lemma}\label{alm. conv}
Let $K \subset \subset X\setminus D$, and $d(\partial K,D)>4\vep$. Then there exists a $\delta = \delta(n,\vep,A,\Lambda)$, such that if $T$ is a neighborhood of $D$ in $X\backslash K$ with $d(p,\partial T) > 2\vep ~\forall ~p\in K$, and $Vol_{d}(\partial T) < \delta$, then, for all $p,q\in K$, there exists a $q^{\prime} \in B_{d}(q,\vep)$ and a minimal $d$-geodesic $\ga^{d}_{pq^{\prime}} : [0,l] \rightarrow X\backslash T$ connecting $p$ to $q^{\prime}$.
\end{lemma}

\begin{proof}

\noindent\textit{\textbf{Claim 1:}} If $\eta<\eta_0(\vep)$, then $$B_{\deta}(q,\vep/2)\subset B_{d}(q,\vep)$$ for all $q\in K$. 

\noindent Suppose not, then for arbitrarily small $\eta$, there exists an $x \in X$ such that $d_{\eta}(q,x)< \vep/2$, but $d(q,x)>\vep$. The minimal $\eta$-geodesic $\ga^{\eta}_{qx}$ has a first point of contact $\tilde x\in \partial B_{d}(q,\vep)$. Then $\sL_{\omega}(\ga^{\eta}_{q\tilde x})\geq\vep$, and hence $d_{\eta}(q,\tilde x) = \sL_{\eta}(\ga_{q,\tilde x})>3\vep/4$ if $\eta$ is sufficiently small, by uniform smooth convergence on $X\backslash T$ and the fact that $\ga^{\eta}_{qx}$ is minimal. This is a contradiction and the claim is proved. \\

\noindent{\textit{\textbf{Claim-2:}}} There exists a $\delta = \delta(n,\vep,A,\Lambda)$ such that, for $\eta < \eta_0(\delta)$ small, any $p,q \in K$ and any $T$ satisfying $Vol_{\omega}(\partial T)<\delta$, there exists $q_\eta \in B_{\deta}(q,\vep/2)$ and a minimal unit speed $\oeta$-geodesic $\ga^{\eta}_{pq_\eta} : [0,l_\eta] \rightarrow X\backslash T$. 

\noindent If not, then by volume comparison, diameter bound, Lemma \ref{gromov}, and volume convergence (cf. Lemma \ref{smt conv}), 
$$ c\kappa \vep^{2n} \leq cVol_{\oeta}(B_{\deta}(q,\vep/2)) \leq Vol_{\oeta}(\partial T) \leq 2 Vol_{\om}(\partial T) \leq 2\delta$$
Pick $\delta = c\kappa\vep^{2n}/4$ to get a contradiction.

So there is a sequence of points $q_\eta \in B_{\deta}(q,\vep/2)\subset B_{d}(q,\vep)$ and $\eta$-minimal geodesics $\ga^{\eta}_{pq_\eta} \subset X\backslash T$. Since the convergence is smooth on $X\backslash T$ and the diameter is uniformly bounded, by Ascoli-Arzela there exists a $q^{\prime} \in B_{d}(q,\vep)$ and a limiting geodesic $\ga_{pq^{\prime}} :[0,l]\rightarrow X\backslash T$ from $p$ to $q'$. \\

\noindent{\textbf{\textit{Claim-3:}}}  $\ga_{pq^{\prime}} $ is $d$-minimal. i.e $$\sL_{\omega}(\ga_{pq^\prime})= d(p,q')$$. 

If not, then by definition of $d$, there exists a path $\tilde \ga_{pq^{\prime}} : [0,1] \rightarrow X\setminus D$ such that $\sL_{\omega}(\tilde \ga_{pq^\prime}) < \sL_{\omega}(\ga_{pq^\prime}) - \zeta$, for some $\zeta>0$. For $\eta$ small, $d(q',q_\eta)<\zeta/8$. The minimal $d$-geodesic $\ga^{d}_{q'q_\eta}$ doesn't hit $\partial T$. So once again by smooth convergence $\sL_{\eta}(\ga^{d}_{q'q_\eta}) < \zeta/4$.  On the other hand, for $\eta$ small, $$\sL_{\eta}(\tilde\ga_{pq^\prime}) < \sL_{\omega}(\tilde\ga_{pq^\prime}) + \zeta/8 < \sL_{\omega}(\ga_{pq^{\prime}}) - 7\zeta/8 < \sL_{\eta}(\ga^{\eta}_{pq_\eta}) - 6\zeta/8$$

So the concatenation $ \tilde \ga_{pq^\prime}\cdot\ga^{d}_{q'q_\eta}$ is a path from $p$ to $q_\eta$ with length $\sL_{\eta}(\tilde \ga_{pq^\prime}\cdot\ga^{d}_{q'q_\eta}) <  \sL_{\eta}(\ga^{\eta}_{pq_\eta}) - 6\zeta/8 + \zeta/4 = \sL_{\eta}(\ga^{\eta}_{pq_\eta})  -\zeta/2$, contradicting the minimality of $\ga^{\eta}_{pq_\eta}$. Hence $\sL_{\omega}(\ga_{pq^\prime})= d(p,q')$.
\end{proof}

\bigskip

\begin{proof}[\textbf{Proof of  Proposition \ref{smoothening}}]
Fix a small $\vep>0$, and choose a tubular neighborhood $E$ of $D$ such that $K = X\backslash E$ is $\vep$-dense with respect to the distance $d$ and $Vol(E,\om) < \vep^{4n}$. The proof of the Gromov-Hausdorff convergence is completed in two steps: \\

\noindent\textbf{\textit{Claim-1:}} There exists a $\eta_0 = \eta_0(\vep)>0$ such that $\forall \eta < \eta_0$, $K$ is $\vep$-dense with respect to $\deta$.

\noindent\textit{Proof}.
If not, then there exists a sequence $p_\eta\in E$ such that $B_{\deta}(p_\eta,\vep) \subset E$. Using volume comparison, diameter bound and the fact that volumes of balls converge, for some uniform $\kappa>0$ and $\eta$ small,
\begin{equation*}
\kappa \vep^{2n} < Vol_{\oeta}(B_{\deta}(p_\eta,\vep)) < Vol_{\oeta}(E) < 2Vol_{\om}(E) < 2\vep^{4n}
\end{equation*}
 which is a contradiction if $\vep$ is small.  \\

\noindent\textbf{\textit{Claim-2:}} There exists a $\eta_0 = \eta_0(\vep)>0$ such that $\forall \eta < \eta_0$ and for all $p,q\in K$, $$|\deta(p,q) - d(p,q)| < \vep$$
\noindent{\textit{Proof}}. Let $\tilde \vep = d(\partial K, D)/4$, so that in particular $\tilde \vep < \vep/4$. We first claim that a neighborhood  $T$ of $D$ can be chosen with $Vol(\partial T,\omega)$ arbitrarily small. This can be done because $D$ has real co-dimension two. For a unit polydisc in $\CC^n$ with a model edge metric with cone angle $2\pi\be_j$ along $[z_j=0]$, such a neighborhood can be constructed explicitly. One can then glue together these local neighborhoods to obtain a neighborhood of $D$ in $X$ with the volume of the boundary arbitrarily small. In particular one can construct a $T$ such that $d(\partial T,K)>2\tvep$ and $Vol(\partial T,\omega) < \delta$ where $\delta = \delta(n,\tvep,A,\Lambda)$ is the constant in Lemma \ref{alm. conv}. 

Next, by Lemma \ref{alm. conv}, for all $p,q \in K$ there exists $q' \in B_{d}(q,\tvep)$ and a minimal $d$-geodesic $\ga^{d}_{pq'} \subset X\backslash T$ .  Like in the argument for the proof of Lemma \ref{alm. conv}, for $\eta$ small, $d_{\eta}(q,q')< 2\tvep$. Then by uniform smooth convergence on $X\backslash T$, there exists $\eta_0>0$ such that for $\eta<\eta_0$ and all $p,q \in K$, $$d_{\eta}(p,q) < \sL_{\eta}(\ga^{d}_{pq'}) + \deta(q,q') < \sL_{\omega}(\ga^{d}_{pq'}) +3 \tvep = d(p,q') + 3\tvep < d(p,q) + 4\tvep < d(p,q) + \vep$$\\
On the other hand, recall that $\ga^{d}_{pq'}$ is constructed as the limit of $\eta$-minimal geodesics $\ga^{\eta}_{pq_\eta}\subset X\backslash T$ with $q_\eta \in B_{\deta}(q,\tvep/2)\subset B_{d}(q,\tvep)$, and $q_\eta \rightarrow q^\prime$. So, $$d(p,q) < d(p,q^\prime)+\tvep < \mathcal{L}_{\eta}(\ga^{\eta}_{pq_\eta}) + 2\tvep = d_{\eta}(p,q_\eta) + 2\tvep < d_{\eta}(p,q) + 5\tvep/2 < d_{\eta}(p,q) + \vep$$ This finishes the proof of Claim-2. \\

 We can now complete the proof of the theorem. For small $\eta<\eta_0(\vep)$,
 \begin{eqnarray*}
&&d_{GH}((X,\deta),(X,d)) \\
&\leq& d_{GH}((X,\deta), (K,\deta)) + d_{GH}((K,\deta), (K,d)) + d_{GH}((K,d), (X,d)) \\
&<& 3\vep,
\end{eqnarray*}

\noindent where we use Claim 1 to bound the first term, Claim 2 to bound the second term, while the last term is bounded by $\vep$ from the choice of $K$.  Now, letting $\vep$ go to zero, we see that $(X,d_\eta)$ converges in Gromov-Hausdorff topology to $(X,d)$. 
 
\end{proof}

\section{\bf Estimate on volume density and proof of geodesic convexity}

Theorem \ref{conv} follows from a theorem of Colding-Naber \cite{CN}. In the previous section, we proved that $(X,d)$ is the Gromov-Hausdorff limit of smooth Riemannian metrics. The crucial point in proving geodesic convexity is that the regular set in the sense of Cheeger-Colding \cite{CC1} coincides with $X\setminus D$, and hence is open. To prove this, we need to show that the volume density of balls in $(X,d)$ centered on the divisor is strictly less than the Euclidean volume density. We do this by reducing to the case of a smooth divisor (i.e when $N =1$), and using known regularity results in this situation \cite{JMR,CDS2}. The volume density function of an $m$-dimensional Riemmanian manifold $(M^m,g)$ at a $p\in M$ is defined as $$V_g(p,r) := \frac{Vol_g(B_g(p,r))}{r^m}$$ We first observe the following elementary fact.

\begin{lemma}\label{model}
Let $\ob$ denote the model edge metric on $\CC^n$ with cone angle $2\pi\be$ along $[z_1=0]$ i.e $$ \ob = \frac{\sqrt{-1}}{2}\Big(\be^2|z_1|^{-2(1-\be)}dz_1\wedge d\bar z_1 + \sum_{j=2}^{n}{dz_j\wedge d\bar z_j}\Big) $$
Then for any $r>0$, $$ V_{\ob}(0,r) = \al(n)\be $$
where $\al(n) = \pi^n/n!$ is the volume of the unit Euclidean ball in $\CC^n$.
\end{lemma}

\begin{proof}
The $\ob$-minimal geodesic connecting the origin to any $(z_1,z') := (z_1,z_2,\cdots,z_n) \in \CC^n$ is given by $\ga(t) = (t^{1/\be}z_1,tz_2\cdots,tz_n)$, and it is easily seen that $\mathcal{L}_{\ob}(\ga) = |z_1|^{2\be} + |z'|^2$. So $B_{\ob}(0,r) = \{z\in \CC^n~|~ |z_1|^{2\be} + |z'|^2 < r\}$. But then, using polar coordinates $z_j = \rho_j e^{i\theta_j}$, and the change of variables $u = \rho^{2\be}$ in the third line, 
\begin{align*}
r^{2n}V_{\ob}(0,r) &= \int_{B_{\ob}(0,r)} {\frac{\ob^n}{n!}}\\
&= \be^2 (2\pi)^n\int_{\rho_1^{2\be} + \cdots + \rho_n^2 < r} {\frac{ (\rho_1d\rho_1)(\rho_2d\rho_2)\cdots (\rho_nd\rho_n)}{\rho_1^{2(1-\be)}}}\\
&= \be(2\pi)^n\int_{u^{2} + \cdots + \rho_n^2 < r} {(udu)(\rho_2d\rho_2)\cdots (\rho_nd\rho_n)}\\
&= \be\int_{0}^{2\pi} \cdots \int_{0}^{2\pi}\int_{u^{2} + \cdots + \rho_n^2 < r} {(udud\theta_1)(\rho_2d\rho_2d\theta_2)\cdots (\rho_nd\rho_nd\theta_n)}\\
&= \be \al(n)r^{2n}
\end{align*}
\end{proof}

\begin{lemma}\label{one divisor}
Suppose $D$ is a smooth divisor with defining section $s$, $h$ a smooth Hermitian metric on $[D]$, and $\tomega = \omega+\ddbar \tilde\vp$, a K\"ahler current solving $$\tomega^n = \frac{\tilde \Omega}{|s|_h^{2(1-\be)}}$$ for some smooth volume form $\tilde \Omega$ with $\be \in (0,1)$. Then for any $p\in D$, $$\lim_{r\rightarrow 0}V_{\tomega}(p,r) = \be\al(n)$$
\end{lemma}
\begin{proof}
By Prop. 26 in \cite{CDS2}, for all $\zeta>0$, there exists an $r_\zeta>0$ such that in some holomorphic coordinates centered at $p\in D$, $$(1-\zeta)\ob < \tomega < (1+\zeta)\ob$$ on $B_{\tomega}(p,r_\zeta)$. In \cite{CDS2}, the metrics under consideration are conical K\"ahler-Einstein metrics.  A key technical point is a Liouville's theorem based on the observation that the conical re-scalings of $\tomega$ defined by $\tomega_\ep = \ep^{-2}T_{\ep}^{*}\tomega$, where $T_\ep(z_1,\cdots,z_n) = (\ep^{1/\be}z_1,\ep z_2,\cdots,\ep z_n)$, converge to a metric cone on $\CC^n$. In the present context, by Proposition \ref{smoothening} one can approximate $\tomega$ by smooth metrics with uniform Ricci lower bound. Then the convergence of the re-scalings to a metric cone is a consequence of general results of Cheeger-Colding \cite{CC1}. Now it is easy to see that $$ \Big(\frac{1-\zeta}{1+\zeta}\Big)^n V_{\be}(0,\frac{r}{\sqrt{1+\zeta}}) < V_{\tomega}(p,r) <  \Big(\frac{1+\zeta}{1-\zeta}\Big)^n V_{\be}(0,\frac{r}{\sqrt{1-\zeta}}); \hspace{0.2in} \forall r< r_\zeta$$ Lemma \ref{model} follows by letting $\zeta\rightarrow 0$ and $r\rightarrow 0$.
\end{proof}

\medskip

\begin{proposition}\label{vol est}
There exists an $\zeta > 0$ and $r(\zeta)>0$, such that for any $r<r(\zeta)$, and any $p\in D$, $$V_{d}(p,r):= \frac{Vol(B_{d}(p,r))}{r^{2n}} < (1-\zeta)\al(n)$$ where $\al(n) = \pi^n/n!$ is the volume of the unit Euclidean ball in $\CC^n$. 
\end{proposition}

\begin{proof} The proposition is proved by smoothening out all but one divisor, and using Lemma \ref{one divisor}. Without loss of generality, let $p\in D_1$, and consider the equation 

\begin{equation*}
\begin{cases}
\oep^n = \frac{e^{-\lambda \psi_\ep - f_\ep + c_\ep}\Omega}{|s_1|_{h_1}^{2(1-\be_1)}}\\
\oep = \homega + \ddbar \vpep > 0
\end{cases}
\end{equation*}

\noindent where $\psi_\ep$ is the sequence approximating $\vp$ from section 2, $f_\ep = \log\Big(\prod_{j=2}^{N}{(|s_j|_{h_j}^2+\ep)^{(1-\be_j)}}\Big)$ and $c_\ep$ is a constant such that the integrals match up. By Prop. \ref{smoothening} there exists a sequence $\oepet$ of smooth K\"ahler metrics and constants $A$ and $\Lambda$ such that 

\begin{align*}
&Ric(\oepet) > -A\oepet ~;~ diam(X,\oepet) < \Lambda\\
&\oepet \xrightarrow{C^{\infty}_{loc}(X\setminus D_1)} \oep \\
&(X,\oepet) \xrightarrow{d_{GH}} (X,\oep)
\end{align*}

By the Bishop-Gromov comparison theorem for the metrics $\oepet$ and Colding's volume convergence theorem \cite{Co}, for $r'<r$ $$\frac{V_{\oep}(p,r)}{V_{-A}(\tilde p,r)}\leq \frac{V_{\oep}(p,r')}{V_{-A}(\tilde p,r')}$$ where $V_{-A}(\tilde p ,r)$ is the volume ratio for the space form of constant sectional curvature $-A/(2n-1)$. Taking $r'\rightarrow 0$, by Lemma \ref{one divisor} 
\begin{align*}
V_{\oep}(p,r) &\leq \be_1V_{-A}(\tilde p,r)\\
&\leq \frac{1+\be_1}{2}\al(n)
\end{align*}
if $r< \bar r = \bar r(A)$. Moreover, since the Ricci lower bounds for $\oepet$ are uniform, by an elementary diagonalization argument, $\oep \xrightarrow{d_{GH}} \omega$ as $\ep \rightarrow 0$. Then once again by Coldings theorem on volume convergence, for $r<\bar r$, $$V_{d}(p,r)< \frac{1+\be_1}{2}\al(n)$$ This proves the proposition with $\zeta = \frac{1}{2}max((1-\be_1), \cdots,(1-\be_N))$ and $r(\zeta) = \bar r$.
\end{proof}
 
 \noindent Since $(X,d)$ is the Gromov-Hausdorff limit of $(X,\oeta)$, one can talk about the regular set, in the sense of Cheeger-Colding. It is defined as $$\sR = \{p\in X ~|~ (X, r_j^{-2}d, p) \xrightarrow{d_{GH}} (\mathbb{C}^n,d_{euc},0) \text{ for any sequence } r_j\rightarrow 0\}$$
 
\begin{lemma}
$\sR=X\setminus D$, and is in particular open and dense in $(X,d)$. 
\end{lemma}

\begin{proof} By smooth convergence away from $D$, it is clear that $X\setminus D \subset \mathcal{R}$. On the other hand, suppose $p\in \mathcal{R}$, then by Colding's volume convergence, the volume density $V_{d}(p,r)$ can be made arbitrarily close to the $\al(n)$ for a small enough $r>0$. But then $p$ cannot belong to $D$ since this would contradict with Proposition \ref{vol est}. Hence $\mathcal{R}=  X\setminus D$, and is consequently open.  The denseness follows from the fact that $X\setminus D$ has full measure.
\end{proof}

\begin{proof}[\textbf{Proof of Theorem \ref{conv}}] We follow the line of argument in \cite{CN}.  By Colding and Naber's result on the H\"older continuity of the tangent cones of limiting spaces of sequences with a Ricci lower bound \cite[Cor. 1.5]{CN}, the set of regular points in the interior of a limiting geodesic is closed. On the other hand, by the above lemma, this set is also open. Therefore, as soon as one interior point lies in $X\setminus D$, all must, and the theorem is proved.  
\end{proof}

\begin{proof}[\textbf{Proof of Corollary \ref{smooth geod}}]
For every $\eta>0$, there exists a unit speed $\eta$-minimal geodesic $\ga^{\eta}:[0,l_\eta]\rightarrow X$ connecting $p$ and $q$ with $l_\eta \rightarrow l$. By the Ascoli-Arzela theorem for Gromov-Haudorff limits, there exists a continuous limiting geodesic $\ga:[0,l]\rightarrow X$ connecting $p$ and $q$. By Theorem \ref{conv}, $\ga$ stays away from $D$. For any $\ga(t_0)$ with, there exists a small ball $B_{d}(\ga(t_0),\ep) \subset X\setminus D$. By the argument of Claim-1 in the proof of Lemma \ref{alm. conv} $B_{d_\eta}(\ga(t_0),\ep/2) \subset B_{d}(\ga(t_0),\ep)$ for $\eta$ small enough. By convergence of geodesics and the fact that the geodesics are of unit speed, there exists a $\delta$ such that, for $\eta$ small enough $\ga^\eta(t) \in B_{d}(\ga(t_0),\ep)$ for all $|t-t_0|< \delta$. By smooth convergence of the metrics on $X\setminus D$, it is easily seen that $\ga|_{(t_0-\delta,t_0+\delta)}$ must be smooth, and hence all of $\ga$ must be smooth.
\end{proof}

\section{\bf Comparison theorems for conical metrics along simple normal crossing divisors} 

For this section we fix $D$ to be an effective simple normal crossing divisor given by \eqref{divisor}, and $\omega$ to be a conical K\"ahler metric along $D$ inducing the metric $d$ on $X$. The aim of this section is to present extensions of some classical comparison theorems to this singular setting. The crucial point is that the cut locus has measure zero. This is already proved in \cite{H,CN}. For the convenience of the reader, we offer an elementary proof in the conical case by exploiting smooth convergence away from the divisor. 
\begin{definition}
We say that $$Ric(\omega)>-A\omega$$ if there exists a smooth positive closed (1,1) form $\chi$ such that $$Ric(\omega) + A\omega = \chi + [D]$$ 
\end{definition}

For a point $p\in X\setminus D$, let $$\sE_p = \{v \in T_pX ~|~ \exists \text{ geodesic } \ga:[0,1]\rightarrow X\setminus D \text{ with } \ga(0)=p~,~\ga^{\prime}(0) = v \}$$The exponential map is well defined and smooth on $\sE$.  The following lemma follows directly from Theorem \ref{conv}.

\begin{lemma}
$exp_p : \sE_p \rightarrow X\setminus D$ is surjective. 
\end{lemma}

We define the cut locus and conjugate locus in the usual way.

\begin{definition}
\begin{enumerate}
\item For a $p\in X\setminus D$, the cut locus is defined by  $$\sC_p = \{x\in X ~|~ \forall z\in X\backslash \{x\},~ d(p,x) + d(x,z) > d(p,z)\}$$ 
\item The conjugate locus is defined by $$Conj( p) = \{x\in X\setminus D ~|~ \exists ~v\in exp_p^{-1}(x) \text{ such that } dexp_p \text{ is degenerate at } v \}$$ Furthermore, $x=\ga(t_0)$ is said to be conjugate to $p$ along a unit speed geodesic $\ga:[0,l]\rightarrow X\setminus D$ if $exp_p$ is singular at $v = t_0\ga'(0)$.  
\end{enumerate}
\end{definition}

The following useful characterization of the cut locus from standard Riemannian geometry \cite{CE} also extends to this setting. 

\begin{lemma}\label{cut locus char}
Let $\ga:[0,l]\rightarrow X\setminus D$ be a smooth unit-speed geodesic emanating from $p$. Then $x=\ga(t_0) \in \sC_p$ if and only if one of the following holds at $t=t_0$ and neither holds for any smaller value of $t$ : 
\begin{enumerate}
\item \label{1}$x$ is conjugate to $p$ along $\ga$.
\item \label{2}There exists a unit speed minimal limiting geodesic $\sigma \neq \ga$ connecting $p$ and $x$.
\end{enumerate}
\end{lemma}
\begin{proof}
Suppose $\ga(t_0) \in \sC_p$, $\ep_j\rightarrow 0$ and $\sigma_j:[0,l_j] \rightarrow X\setminus D$ be a unit speed smooth limiting minimal geodesic connecting $p$ to $x_j=\ga(t_0 + \ep_j)$.  By continuity of the distance function, $l_j\rightarrow t_0$. By the same argument as in the proof of Cor. \ref{smooth geod}, one can show that there exists a $\sigma:[0,t_0] \rightarrow X\setminus D$ connecting $p$ and $x$ such that $\sigma_j \rightarrow \sigma$ smoothly.  If $\sigma\neq\ga$, criteria \eqref{2} is satisfied. If not, then arbitrarily small neighborhoods of $t_0\ga^{\prime}(0)$ in $\sE_p$ have two distinct vectors, namely $l_{j}\sigma_{j}^{\prime}(0)$ and $(t_0+\ep_j)\ga^{\prime}(0)$, mapped to the same point $x_j$ under the exponential map. By the inverse function theorem, $t_0\ga^{\prime}(0)$ is a singular point of $\exp_p$ or equivalently $x$ is a conjugate point along $\ga$.
\end{proof}

As an immediate corollary we have 

\begin{corollary}
$\sC_p$ has measure zero with respect to $\omega$.
\end{corollary}

\begin{proof}
From the previous Lemma $\sC_p \subset \{\text{singular values of } exp_p\}\cup \{r \text{ is not differentiable}\}$.   The first one has measure zero by Sard's theorem, while the second one has measure zero because $r$ is Lipshitz.
\end{proof}

We now present some classical comparison theorems. We also recall the proofs to emphasize that geodesic convexity, even of the slightly weaker kind proved in the present article, is all that is needed for the extensions to the conical setting.

\begin{theorem}[Laplacian comparison] Suppose $Ric(\omega) >(2n-1) \la\omega$ for some $\la \in \mathbb{R}$, and $\tilde X$ is the $2n$-dimensional space form with constant sectional curvature $\la$. Let $r(x)$ and $\tilde r(\tilde x)$ be distance functions to some fixed points in $X$ and $\tilde X$ respectively. Then for any $x\in X\setminus D$ where $r$ is smooth, and any $\tilde x \in \tilde X$ where $\tilde r$ is smooth with $r(x) = \tilde r(\tilde x)$, $$\Delta r(x) \leq \tilde\Delta\tilde r(\tilde x)$$
\end{theorem}
\begin{proof}
By Bochner formula,
\begin{align*}
0 &= |\nabla^2 r|^2 + \frac{\partial(\Delta r)}{\partial r} + Ric(\nabla r,\nabla r)\\
&\geq (\Delta r)^2 +  \frac{\partial(\Delta r)}{\partial r} + (2n-1)\la
\end{align*}
Note that equality holds in the case of $\tilde X$. So, if $\ga\subset X\setminus D$ and $\tilde \ga$ are unit speed minimal geodesics joining the reference points to $x$ and $\tilde x$ respectively, then $u(t) = \Delta r(\ga(t)) - \tilde \Delta \tilde r(\tilde \ga(t))$ satisfies the differential inequality $$\dot u + gu \leq 0$$ where $g = \Delta r(\ga(t)) + \tilde \Delta \tilde r(\tilde \ga(t))$. Moreover $$\lim_{t\rightarrow 0}|\Delta r(\ga(t))-\Big( \frac{2n-1}{t}\Big)| = \lim_{t\rightarrow 0} |\tilde \Delta \tilde r(\tilde \ga(t))-\Big( \frac{2n-1}{t}\Big)| = 0$$ i.e $u(0) = 0$. By the method of integrating factors for first order ODEs, it is easily seen that $u(t) \leq 0  ~\forall ~ t$. 

\end{proof}

\begin{theorem}[Myer's theorem]\label{Myers}
With $D$ as above, suppose $\omega$ is a conical K\"ahler metric along $D$ satisfying $Ric(\omega)>(2n-1)\la\omega$ for some $\la>0$. Then $$diam(X,d) < \frac{\pi}{\sqrt{\la}}$$ 
\end{theorem}

\begin{proof}
By explicit calculation, if $\la>0$, and $\tilde X$ is the space form with sectional curvature $\la$, then along a unit speed minimal geodesic $\tilde \ga$, $$\tilde\Delta\tilde r(\tilde \ga(t)) = (2n-1)\sqrt{\la}\frac{ \cos(\sqrt{\la}t)}{\sin(\sqrt{\la}t)}$$ Fix a point $p\in X\setminus D$. For any other point $x\in X\setminus D$, if $\ga$ is the minimal unit speed geodesic joining them, then $$ \Delta r(\ga(t)) \leq (2n-1) \sqrt{\la}\frac{ \cos(\sqrt{\la}t)}{\sin(\sqrt{\la}t)} $$ Since right hand side goes to $-\infty$ as $t\rightarrow \pi/\sqrt{\la}$, $t$, and hence the length of $\ga$, can be at most $\pi/\sqrt{\la}$.
\end{proof}

Next, the exponential map is a diffeomorphism from an open subset of $\sE_p$ onto $X\backslash (D\cup C_p)$. Moreover, since $C_p\cup D$ has measure zero, standard arguments as in \cite{SY} can be used to prove  the Bishop-Gromov  volume comparison. 
\begin{theorem}[Bishop-Gromov volume comparison]\label{BG}
If $Ric(\omega) >(2n-1) \la\omega$ for some $\la\in \mathbb{R}$ and $\tilde X$ is the $2n$-dimensional space form with constant sectional curvature $\la$. Then
\begin{enumerate}
\item If $K\subset X\setminus D$ is any star convex set  centered at $x$, then for $0<r_1<r_2 (< \pi/\sqrt{\la}$ if $\la>0$), $$ \frac{Vol(B_{d}(x,r_2)\cap K) - Vol(B_{d}(x,r_1)\cap K)}{\tilde V(r_2) - \tilde V(r_1)} \leq  \frac{Vol(\partial B_{d}(x,r_1)\cap K)}{Vol(\partial \tilde B (r_1))}$$
where $\tilde B( r )$ is a ball of radius $r$ in $\tilde X$ and $\tilde V( r ) =  Vol (\tilde B ( r ))$.
\item For all $x\in X$, the volume ratio $$\frac{Vol(B_d(x,r))}{\tilde V( r)}$$ is non-increasing in $r$. 
\end{enumerate}
\end{theorem}

\begin{remark}\label{con gromov}
As a corollary to Theorem \ref{BG} above, Lemma \ref{gromov} generalizes to K\"ahler currents satisfying equation \eqref{con ke}, and in particular to conical K\"ahler-Einstein metrics. This was very useful in \cite{DGSW} to study the degeneration of conical K\"ahler-Einstein metrics on toric manifolds.\end{remark}

\end{document}